\documentclass [a4paper,12pt, reqno]{amsart}
\usepackage{longtable}
\usepackage{hyperref}
\usepackage[T1]{fontenc}
\usepackage[utf8]{inputenc}
\usepackage[english]{babel}
\usepackage{textcomp}
\usepackage{dsfont}
\usepackage{latexsym}
\usepackage{amssymb}
\usepackage{amsthm}
\usepackage{amsmath}
\DeclareMathAlphabet{\mathpzc}{OT1}{pzc}{m}{en}
\usepackage{yfonts}
\usepackage{xfrac}
\usepackage{newlfont}
\usepackage{graphicx}
\usepackage{mathtools}
\usepackage{comment}
\usepackage{indentfirst}
\usepackage{braket}
\usepackage{mathrsfs}
\usepackage{xcolor}

\usepackage{etoolbox}
\apptocmd{\lim}{\limits}{}{}
\apptocmd{\sup}{\limits}{}{}
\apptocmd{\inf}{\limits}{}{}
\apptocmd{\liminf}{\limits}{}{}
\apptocmd{\limsup}{\limits}{}{}

\usepackage{geometry}
\usepackage{scalerel}[2014/03/10]
\usepackage[usestackEOL]{stackengine}
\newcommand{\dashint}{\,\ThisStyle{\ensurestackMath{%
			\stackinset{c}{.2\LMpt}{c}{.5\LMpt}{\SavedStyle-}{\SavedStyle\phantom{\int}}}%
		\setbox0=\hbox{$\SavedStyle\int\,$}\kern-\wd0}\int}

\DeclareMathOperator{\card}{Card}

\DeclareMathOperator{\supp}{Supp}

\DeclareMathOperator{\diam}{diam}

\newcommand{\ee}{\mathrm{e}}

\newcommand{\loc}{\mathrm{loc}}

\newcommand{\dd}{\mathrm{d}}

\DeclarePairedDelimiter{\abs}{\lvert}{\rvert}

\DeclarePairedDelimiter{\norm}{\lVert}{\rVert}

\let\originalleft\left
\let\originalright\right
\renewcommand{\left}{\mathopen{}\mathclose\bgroup\originalleft}
\renewcommand{\right}{\aftergroup\egroup\originalright}

\newcommand{\N}{\mathbb{N}}
\newcommand{\Z}{\mathbb{Z}}
\newcommand{\Q}{\mathbb{Q}}

\newcommand{\C}{\mathbb{C}}

\newcommand{\R}{\mathbb{R}}

\newcommand{\Ff}{\mathfrak{F}}

\newcommand{\Mf}{\mathfrak{M}}

\newcommand{\Lc}{\mathcal{L}}
\newcommand{\cM}{\mathcal{M}}

\newcommand{\Pc}{\mathcal{P}}

\newcommand{\meg}{\leqslant}
\newcommand{\Meg}{\geqslant}
\newcommand{\eps}{\varepsilon}
\renewcommand{\phi}{\varphi}
\newcommand{\mi}{\mu}

\newcommand{\leftexp}[2]{{\vphantom{#2}}^{#1}{#2}} 
\newcommand{\trasp}{\leftexp{t}}

\author{Mattia Calzi, Elena Rizzo}
\title{Vector-Valued Singular Integrals on Locally Doubling Spaces}

\date{}

\address{Dipartimento di Matematica, Universit\`a degli Studi di	Milano, Via C. Saldini 50, 20133 Milano, Italy}
\email{{\tt mattia.calzi@unimi.it}, {\tt elena.rizzo@unimi.it}} 

\thanks{{\em Math Subject Classification 2020}: 46G10, 47B90.}
\thanks{The authors are members of the 	Gruppo Nazionale per l'Analisi
	Matematica, la Probabilit\`a e le	loro Applicazioni (GNAMPA) of
	the Istituto Nazionale di Alta Matematica (INdAM). The authors were partially funded by the INdAM-GNAMPA Project CUP\_E5324001950001.
} 

\begin{document}

\theoremstyle{definition}
\newtheorem{deff}{Definition}

\newtheorem{oss}[deff]{Remark}

\newtheorem{ass}[deff]{Assumptions}

\newtheorem{nott}[deff]{Notation}

\newtheorem{ex}[deff]{Example}

\theoremstyle{plain}
\newtheorem{teo}[deff]{Theorem}

\newtheorem{lem}[deff]{Lemma}

\newtheorem{prop}[deff]{Proposition}

\newtheorem{cor}[deff]{Corollary}
\maketitle

\begin{abstract}
	We prove vector-valued boundedness of (suitable) Calder\'on-Zygmund operators and of the (truncated) Hardy--Littlewood maximal function on a \emph{connected} locally doubling metric measure space.
\end{abstract}

\section{Introduction}

Given a metric (or quasi-metric) measure space $(X,d, \mi)$ which satisfies the doubling property, that is, such that there is a constant $C>0$ such that
\begin{equation}\label{eq:1}
\mi(B(x,2r))\meg C\mi(B(x,r))
\end{equation}
for every $x\in X$ and for every $r>0$, it is known since the work of Coifman and Weiss~\cite{CW} that one may develop a reasonably complete theory of Calder\'on--Zygmund operators as in the classical setting of $\R^n$. In particular, see~\cite{Grafakos} for some results on  vector-valued Calder\'on--Zygmund operators.

It is also known that this kind of theory is of a `global' character, in a certain sense, since the Calder\'on--Zygmund decomposition, on which it is based, is of a somewhat non-local nature.
Consequently, if one has to deal with a space which is only locally doubling, that is, a space for which~\eqref{eq:1} holds only for $r\in (0, R]$ for some fixed $R>0$, then the usual Calder\'on--Zygmund theory cannot be applied, even though some extensions of certain aspects of the theory have been studied in a   variety of different settings.  

If the metric space $(X,d)$ is geodesic, that is, if any two points in $ X$ may be joined by a minimizing curve, then one may show that every ball of radius $\meg R$ is a doubling metric measure space, with the induced distance and measure, provided that~\eqref{eq:1} holds for $r\in (0,R]$ (cf.~Remark~\ref{oss:1}). If $(X,d)$ is ultrametric, instead, then it is even easier to show that the balls with radius $\meg 2R$ are doubling (with the same constants).
In both cases, one may therefore apply the usual Calder\'on--Zygmund theory `locally' on each ball, with uniform constants, and then reconstruct a global theory, up to some extent. 

Nonetheless, if $(X,d)$ is neither geodesic nor ultrametric, then it is unclear whether it contains any (non-empty) doubling open subspace at all. The purpose of this note is to address this issue and show that some `local' analogue of the Calder\'on--Zygmund theory may be developed in this context. We shall limit ourselves to proving a local analogue of the main theorems on the continuity of vector-valued Calder\'on--Zygmund operators (Theorem~\ref{teo:1} and Corollaries~\ref{cor:1b} and~\ref{cor:2}), as well as the boundedness of the vector-valued truncated Hardy--Littlewood maximal function (cf.~Theorem~\ref{teo:2}). 
Even though we shall keep track of the requirements on the largest radii for which the doubling condition~\eqref{eq:1} should hold in order for our results to apply, we claim no optimality in this regard; it is in fact possible that a finer analysis may achieve smaller radii on these requirements. This choice is motivated by the fact that locally doubling spaces usually satisfy~\eqref{eq:1} with a constant $C_r$ depending on $r$, in such a way that the mapping $r\mapsto C_r$ is increasing but everywhere finite.
Because of the application (cf.~\cite{BesovTL})  that originally motivated this note, instead of considering $\ell^q(\N)$-valued functions for the vector-valued boundedness of the Hardy--Littlewood maximal function, we consider more general mixed-norm spaces. Apart from that, we follow closely the classical references on the topic (cf.~\cite{CW, Grafakos}), except fot the fact that we work on a metric space instead of a quasi-metric one, for simplicity, and   that we need to assume that the space $X$  be connected   in order to establish a suitable analogue of Whitney's lemma. 
Even though the proofs are relatively simple modifications of the classical ones, we address some subtleties here and there and amend some minor mistakes in some of the classical proofs (cf.~Corollary~\ref{cor:1b} and Theorem~\ref{teo:2}).

\section{Statement and Proof of the Main Results}

We fix a locally compact, \emph{connected}\footnote{We need connectedness only to prove Lemma~\ref{lem:2b}. In fact, one may get away with a weaker assumption, and simply require that for every non-empty, proper, bounded, open subset $U$ of $X$ there is $x\in X\setminus U$ such that $\diam(U)=\diam(U\cup \Set{x})$.} metric space $(X,d)$, endowed with a non-zero positive \emph{inner regular} Radon measure $\mi$ such that $\supp \mi=X$.
For every $R>0$, we denote with $D_R$ the least element of $[1,+\infty]$ such that  
\[
\mi(B(x,2 r))\meg D_R \mi(B(x,r))
\]
for every $x\in X$ and for every $r\in (0,R]$.
Notice that, since we assume that $\supp \mi=X$, one has $\mi(B(x,r))>0$ for every $x\in X$ and for every $r>0$, so that $D_R$ is always well defined (finite or not).

Notice that $\mi$ need not be outer regular, unless $X$ is separable.
In addition, recall that an inner regular Radon measure $\mi$ on $X$ is a Borel measure such that:
\begin{itemize}
	\item $\mi$ is inner regular, that is, for every Borel subset $B$ of $X$, $\mi(B)$ is the least upper bound of the set of $\mi(K)$, where $K$ runs through the set of compact subsets of $B$;
	
	\item every $x\in X$ has an open neighbourhood with finite measure. Since $X$ is locally compact, this is equivalent to saying that every compact subset of $X$ has finite measure.
\end{itemize}

One may then extend $\mi$ to a measure on the $\sigma$-algebra of $\mi$-measurable sets, that is, on the collection of subsets $E$ of $X$ such that for every compact subset $K$ of $X$ there are Borel sets $B,N$ such that $B\subseteq E\cap K\subseteq B\cup N$ and $\mi(N)=0$. 
Notice that we do not assume that $\mi$ be $\sigma$-finite.

\begin{oss}
	If we do not assume that $X$ be locally compact, we may still consider a positive \emph{inner regular} Radon measure $\mi$ on $X$. If $D_R<\infty$ for some $R>0$, then one may show that every ball $B(x,r)$, $r\meg R$, is precompact (and has finite measure). As a consequence, the completion $\widehat X$ of $X$ is locally compact, and $\mi$  induces a positive Radon measure on $\widehat X$ with the same doubling constants. In other words, there is no substantial loss in generality in assuming that $X$ be locally compact.
\end{oss}

\begin{lem}\label{lem:0}
	Take $R>0$ so that $D_R<\infty$. Then $\mi(U), \mi(B(x,r))\in (0,+\infty)$ for every non-empty open subset $U$ of $ X$ with $\diam(U)<2R$, for every $x\in X$, and for every $r\in (0,2R]$.
\end{lem}

\begin{proof}
	Observe first that, for every $x\in X$, there is $r\in (0,R]$ such that $\mi(B(x,r))\in (0,\infty)$, since $\mi$ is a Radon measure. The same then holds for $\mi(B(x,r'))$ for every $r'\in (0,r]$, since $\supp\mi=X$, hence for $\mi(B(x,2^k r'))$ for every $k\in\Z$ and for every $r'\in (0,r]$ such that $2^k r'\meg 2 R$, since $D_R$ is finite. Consequently, $\mi(B(x,r))\in (0,+\infty)$ for every $x\in X$ and for every $r\in (0,2R]$.  Now, let $U$ be a non-empty open subset of $X$ with $\diam(U)< 2R$. Since $U\subseteq B(x,2R)$ for every $x\in U$, it is clear that $\mi(U)<\infty$. Since $\supp \mi=X$, it is clear that $\mi(U)>0$.
\end{proof}

\begin{oss}\label{oss:1}
	Assume that $X$ is geodesic and that $D_R$ is finite for some $R>0$. Then, the restriction of $\mi$ to $B(x,R)$ is doubling for every $x\in X$, with doubling constant $\meg 3 D_R$.
	
	To see this, fix $x\in X$ and $r\in (0,R]$, and take $y\in B(x,r)$ and $r'\in (0,2r]$. Take a minimizing geodesic $\gamma\colon [0,1]\to X$ with $\gamma(0)=x$ and $\gamma(1)=y$. 
	If $d(x,y)\meg r'/2 (\meg r)$, then $B(x,r'/2)\subseteq B(y,r')\cap B(x,r)$, so that
    \begin{align*}
        \mi(B(y,r')\cap B(x,r))
        &\Meg \mi(B(x,r'/2))\\
        &\Meg D_R^{-3}\mi(B(x,\min(4r',r)))\\
        &\Meg D_R^{-3}\mi(B(y,2 r')\cap B(x,r)).
    \end{align*}
     If, otherwise, $d(x,y)>r'/2$, then we may  choose  $t\in (0,1)$ in such a way that $d(\gamma(t),y)= r'/2$, so that
     \begin{equation*}
         d(\gamma(t), x)=d(x,y)-d(\gamma(t),y)= d(x,y)-r'/2< r-r'/2.
     \end{equation*}
     Then, $B(\gamma(t), r'/2)\subseteq B(y,r')\cap B(x,r)$, so that
     \begin{align*}
         \mi(B(y,r')\cap B(x,r))
         &\Meg \mi(B(\gamma(t),r'/2))\\
         &\Meg D_R^{-3}\mi(B(\gamma(t),\min(4r',2r)))\\
         &\Meg D_R^{-3}\mi(B(y,2 r')\cap B(x,r)).
     \end{align*}
     Thus, $B(x,r)$, with the induced structure, is a doubling metric measure space with doubling constant $\meg D_R^3$.
\end{oss}

We proceed with some covering lemmas (cf.~\cite[Theorems 1.2 and 1.3 of Chapter III]{CW} for the doubling case).

\begin{lem}\label{lem:1b}
	Take $R>0$ so that $D_R<\infty$.
	Let $E$ be a   subset of $X$ with $\diam E<2R$, and let $r\colon E\to (0,R]$ be a function.  Then, there is a countable subset $N$ of $E$ such that the $B(x,r(x))$, $x\in N$, are pairwise disjoint, while the $B(x,3 r(x))$, $x\in N$, cover $E$. 
\end{lem}

\begin{proof} 
	We construct by induction a (possibly finite) sequence $(x_j)_j$ of elements of $E$ as follows. Assume that the $x_{j'}$, $j'<j$ have been constructed in such a way that $B(x_{j'},r(x_{j'}))$ are pairwise disjoint, that
    \begin{equation*}
        E_{j'}\coloneqq E\setminus \Big(\bigcup_{j''<j'} B(x_{j''},3 r(x_{j''}))\Big)
    \end{equation*}
    is not empty,  and that $r(x_{j'})>\frac 1 2 \sup\{r(x)\colon x\in E_{j'}\}$. 
	If $E_j=\emptyset$, that is, if the $B(x_{j'},3r(x_{j'}))$ cover $E$, then we stop. If, otherwise, $E_j\neq \emptyset$, then we choose $x_j\in E_j$ such that $r(x_j)>\frac 1 2 \sup\{r(x)\colon x\in E_{j}\}$.
	Now, observe that, if $j'<j$, then $x_j\in E_{j}\subseteq E_{j'}$, so that
    \begin{equation*}
        d(x_j, x_{j'})\Meg 3 r(x_{j'}) \text{ and } r(x_{j'})>\frac 1 2 r(x_j)
    \end{equation*}
    by the inductive assumption. Then,
    \begin{equation*}
        B(x_j, r(x_j))\cap B(x_{j'}, r(x_{j'}))=\emptyset \text{ for every } j'<j.
    \end{equation*}
	Now we let $N$ be the set of the $x_j$. It is clear by construction that the $B(x,r(x))$, $x\in N$, are pairwise disjoint. Moreover, if $N$ is finite, again by construction it is immediate that the $B(x, 3r(x))$, $x\in N$ cover $E$. On the other hand, if $N$ is infinite, let us prove by contradiction that $r(x_j)\to 0$ for $j\to \infty$. Assume that there are $\eps\in (0,2R-\diam(E)]$ and an infinite subset $N'$ of $N$ such that $r(x)\Meg \eps$ for every $x\in N'$. Take $k\in \N$ so that  $\eps> R/2^k$ and fix $\tilde x\in E$. Thus,
	\[
	\mi(B(x,\eps))\Meg \mi(B(x, R/2^k)) \Meg D_R^{-k-1} \mi(B(x, 2R))\Meg D_R^{-k-1} \mi(B(\tilde x,\eps)) 
	\]
	for every $x\in N'$, so that
    \begin{equation*}
        \mi(B(\tilde x, 2R))\Meg \sum_{x\in N'} \mi(B(x,\eps))=+\infty,
    \end{equation*}
    which contradicts  Lemma~\ref{lem:0}. Therefore, $r(x_j)\to 0$ for $j\to \infty$.
	Consequently, if by contradiction there is $x'\in E\setminus \smash{\big(\bigcup_{x\in N} B(x,3 r(x))\big)}$, then there is $j\in \N$ such that $r(x_j)<r(x')/2$. However, since $x'\in E\setminus \smash{\big(\bigcup_{x\in N} B(x,3 r(x))\big)\subseteq E_j}$, we also have that
    \begin{equation*}
        \sup_{x\in E_j}r(x) \Meg r(x')>2 r(x_j),
    \end{equation*}
    which contradicts our choice of $x_j$.
\end{proof}

\begin{lem}	\label{lem:2b} 
	Take $R>0$ so that $D_R<\infty$.
	Let $U$ be a \emph{proper}\footnote{That is, $U\neq X$. This is automatically satisfied if $\diam(X)\Meg R$ (which is a natural assumption, since otherwise $\mi$ is doubling).} \emph{open} subset of $X$ with  $\diam(U)<R$. Then, there are a countable subset $N$ of $U$ and a function $r\colon N\to (0,R/2]$ such that
    \begin{equation*}
        \chi_U\meg \sum_{x\in N} \chi_{B(x,r(x))}\meg D_R^5\chi_U,
    \end{equation*}
    the $B(x,r(x)/3)$, $x\in N$, are pairwise disjoint, and $B(x,\kappa r(x))\setminus  U\neq \emptyset$ for every $x\in N$ and for every $\kappa>2$.
\end{lem}

\begin{proof}
	We may assume that $U\neq \emptyset$.
	We apply Lemma~\ref{lem:1b} to the set $U$ and the function $r'\colon U\to (0,R]$ defined as $r'(x)=\frac{1}{6} d(x, X\setminus U)$ for every $x\in U$.
    Observe that, since
	$X$ is connected and $U\neq \emptyset, X$,  there is $\tilde x\in \overline U\setminus U$. It follows that
    \begin{equation*}
        d(x,X\setminus U)\meg d(x,\tilde x)\meg R
    \end{equation*}
    for every $x\in U$, so that $r'(x)\in (0,R/6]$ for every $x\in U$. Then, by Lemma~\ref{lem:1b}, there is a countable set $N\subseteq U$ such that the $B(x,r'(x))$, $x\in N$, are pairwise disjoint, while the $B(x,3r'(x))$, $x\in N$, cover $U$. We set
    \begin{equation*}
        r(x)\coloneqq 3 r'(x)=\frac{1}{2} d(x, X\setminus U)\in (0,R/2]
    \end{equation*}
    for every $x\in N$, so that the balls $B(x,r(x))$, $x\in N$, cover $U$. It is also clear that $r(x)<d(x,X\setminus U)$, so that $B(x,r(x))\subseteq U$ for every $x\in N$. In addition, if $\kappa>2$ and $x\in N$, then $B(x,\kappa r(x))=B(x,3\kappa r'(x))\not \subseteq U$ by the definition of $r'(x)$, since $3\kappa>6$. 
	Next, observe that, if $x\in U$ and $N_x\coloneqq \Set{x'\in N\colon d(x',x)<r(x')}$, then, for every $x'\in N_x$,
    \begin{equation*}
        2r(x')=6 r'(x')\meg d(x, X\setminus U)+ d(x,x')<d(x, X\setminus U)+ r(x'),
    \end{equation*}
    so that $r(x')< d(x,X\setminus U)$ and then $N_x\subseteq B(x, d(x,X\setminus U))$. Analogously,
    \begin{equation*}
        6 r'(x)=d(x, X\setminus U)\meg d(x,x')+ 6 r'(x')< 9 r'(x'),
    \end{equation*}
    so that $B(x,6 r'(x))\subseteq B(x', 6 r'(x)+ r(x'))\subseteq B(x', 12 r'(x'))$ for every $x'\in N_x$. Consequently,   
	\[
	\begin{split}
		\mi(B(x,12 r'(x)))&\Meg \sum_{x'\in N_x} \mi(B(x',r'(x')))\\
		&\Meg D_R^{-4}\sum_{x'\in N_x} \mi(B(x',12 r'(x')))\\
		&\Meg \card(N_x) D_R^{-4} \mi(B(x,6r'(x)))\\
		&\Meg \card(N_x) D_R^{-5} \mi(B(x,12 r'(x))),
	\end{split}
	\]
	so that $\card(N_x)\meg D_R^5$ since $\mi(B(x,12r'(x)))\in (0,+\infty)$ by Lemma~\ref{lem:0}.	
\end{proof}

\begin{deff}
	Let $B$ be a Banach space.
	For every $R>0$, we define 
	\[
	(\cM_R f)(x)\coloneqq \sup_{\substack{y\in X, r\in (0,R]\\ x\in B(y,r) }} \dashint_{B(y,r)} \abs{f}\,\dd \mi
	\]
	and
	\[
	(\widetilde \cM_R f)(x)\coloneqq \sup_{r\in (0,R]} \dashint_{B(x,r)} \abs{f}\,\dd \mi
	\]
	for every $\mi$-measurable function $f\colon X\to B$ and for every $x\in X$, with the convention $\frac{\infty}{\infty}=0$. 
\end{deff}
Recall that $f$ is $\mi$-measurable if and only if for every compact subset $K$ of $X$ and for every $\eps>0$ there is a compact subset $K'$ of $K$ such that $\mi(K\setminus K')<\eps$ and such that the restriction of $f$ to $K'$ is continuous. In particular, also $\abs{f}$ is $\mi$-measurable. We are therefore considering `strong' measurability and `Bochner' integrals (cf.~\cite{DS,DU}).

\begin{oss}
	Take $R>0$ and a $\mi$-measurable function $f\colon X\to B$ for some Banach space $B$. Then,
	\[
	\widetilde \cM_R f\meg \cM_R f\meg D_R\widetilde \cM_{2R} f.
	\]
\end{oss}

We now prove the boundedness of the maximal function $\cM_R$, which will be crucial in establishing a somewhat `local' Calder\'on--Zygmund decomposition (cf.~\cite[Theorem 2.1 and Corollary 2.3 of Chapter III]{CW} for the doubling case).

\begin{lem}\label{lem:3b}
	Take $R>0$ so that $D_{3R}<\infty$.
	The following hold:
	\begin{enumerate}
		\item[\textnormal{(1)}] for every $\mi$-measurable function $f\colon X\to B$, $\cM_R f$ is a lower semi-continuous function on $X$;
		
		\item[\textnormal{(2)}] for every $\mi$-measurable set $E\subseteq X$ with $\diam(E)<6R$, for every $\mi$-measurable function $f\colon X\to B$ and for every $\lambda>0$,
		\[
		\mi\Bigl( \Set{x\in E\colon (\cM_R f)(x) > \lambda } \Bigr)\meg \frac{D_{ 3R}^4 \norm{f}_{L^1( \mi; B)}}{\lambda};
		\] 
		
		\item[\textnormal{(3)}] for every $\mi$-measurable set $E\subseteq X$ with $\diam(E)<6R$, for every $\mi$-measurable function $f\colon X\to B$, and for every $p\in (1,\infty]$,
        \begin{equation*}
            \norm{\chi_E \cM_R f}_{L^p(\mi)}\meg 2 p'^{1/p} D_{ 3R}^{4/p}   \norm{f}_{L^p( \mi;B)};
        \end{equation*}
		
		\item[\textnormal{(4)}]  for every $f\in L^1_\loc(\mi;B)$,
        \begin{equation*}
            \dashint_{B(x,r)} f\,\dd \mi\to f(x)
        \end{equation*}
        for $\mi$-almost every $x\in X$.\footnote{Notice that we assume $\mi$ to be inner regular. Since it is then naturally to call $\mi$-negligible a $\mi$-measurable set $N$ such that $\mi(N)=0$, we shall \emph{not} follow the terminology of~\cite{BourbakiInt2}, according to which those sets should be called \emph{locally} $\mi$-negligible.}
	\end{enumerate}
\end{lem}
 
Observe that (2) and (3) may actually be improved removing the assumption on $\diam(\supp f)$. We shall not prove this fact here, but only in Theorem~\ref{teo:2}.

\begin{proof} 
	(1) Take $\lambda\Meg 0$ and take $x\in X$ such that $(\cM_R f)(x)>\lambda$. Then,  there are $y\in X$ and $r\in (0,R]$ such that $x\in B(y,r)$ and $\dashint_{B(y,r)} \abs{f}\,\dd \mi>\lambda$, so that $(\cM_R f)(x')>\lambda$ for every $x'\in B(y,r)$. The assertion follows.
	
	(2) Take $\lambda>0$, and observe that, by (1),
    \begin{equation*}
        U_\lambda\coloneqq \Set{x\in X\colon (\cM_R f)(x)>\lambda}
    \end{equation*}
    is an open subset of $X$. 
	If $x\in U_{\lambda}$, then there are $y_x\in X$ and  $r_x\in (0,R]$ such that $x\in B(y_x, r_x)$ and $\dashint_{B(y_x,r_x)}\abs{f}\,\dd \mi>\lambda$. If we set $r(x)\coloneqq 2 r_x\in (0,2R]$, then
	\[
	\dashint_{B(x,r(x))} \abs{f}\,\dd \mi\Meg \frac{1}{D_{ 3R}^2}\dashint_{B(y_x,r_x)} \abs{f}\,\dd \mi>\frac{\lambda}{D_{ 3R}^2}
	\]  
	since
    \begin{equation*}
        \mi(B(x,r(x)))\meg D_{ 3R} \mi(B(x,r_x))\meg D_{ 3R}\mi(B(y_x,2r_x))\meg  D_{ 3R}^2 \mi(B(y_x,r_x)).
    \end{equation*}
    Then, by Lemma~\ref{lem:1b} we may find a countable subset $N$ of $E\cap U_{\lambda}$ such that the $B(x,r(x))$, $x\in N$, are pairwise disjoint, while the $B(x,3r(x))$, $x\in N$, cover $E\cap U_{\lambda}$. Then,
	\[
	\begin{split}
		\norm{f}_{L^1(\mi;B)}&\Meg \sum_{x\in N} \int_{B(x,r(x))} \abs{f}\,\dd \mi\\
		& \Meg \frac{\lambda}{D_{ 3R}^2} \sum_{x\in N} \mi(B(x,r(x)))\\
		&\Meg  \frac{\lambda}{D_{ 3R}^4}\sum_{x\in N} \mi(B(x,3 r(x)))\\
		&\Meg  \frac{\lambda}{D_{ 3R}^4} \mi(E\cap U_{\lambda})
	\end{split}
	\]
	whence the conclusion by the arbitrariness of $\lambda$.
	
	(3) This follows by interpolation, since clearly $(\cM_R f)(x)\meg \norm{f}_{L^\infty(\mi;B)}$ for every $x\in X$ (cf.~\cite[Theorem 1.3.2]{GrafakosClassical}).
	
	(4) Since the assertion is local,\footnote{Notice that, by inner regularity, a set $N$ is $\mi$-negligible if and only if $N\cap B(x,R)$ is $\mi$-negligible for every $x\in X$ even if $X$ need not be separable.} we may assume that $f\in L^1(\mi;B)$ and that $\diam(\supp f)<6 R$. By~\cite[Remark 3 to Proposition 3 of Chapter IX, \S\ 5, No.\ 1]{BourbakiInt2},   there is a sequence $(g_j)$ of continuous functions from $ X$ into $ B$ such that $\norm{f-g_j}_{L^1(\mi;B)}<2^{-j}$ for every $j\in\N$, and such that all the $g_j$ are supported in a fixed set with diameter $<6R$. We may also assume that $g_j(x) \to f(x)$ for $\mi$-almost every $x\in X$. Therefore, 
	\[
	\begin{split}
		&\limsup_{r\to 0^+}\abs*{\dashint_{B(x,r)} f\,\dd \mi-f(x) }\\
        &\quad\meg \limsup_{r\to 0^+} \dashint_{B(x,r)} \!\abs{f-g_j}\,\dd \mi+ \abs{f(x)-g_j(x)}+\limsup_{r\to 0^+}\abs*{\dashint_{B(x,r)} g_j\,\dd \mi-g_j(x) }\\
		&\quad\meg \cM_R (f-g_j) (x)+ \abs{f(x)-g_j(x)}
	\end{split}
	\]
	for every $x\in X$ and for every $j\in \N$. The assertion follows by means of (2).
\end{proof}

\begin{lem}\label{lem:4b}
	Let $\kappa\in (2,3]$ and $R>0$ so that $D_{3\kappa R}<\infty$. Let $E$ be a $\mi$-measurable subset of $X$ with $\mi(E)>0$ and $\diam(E)< R$ and take $f\in~L^1(\mi;B)$ concentrated in $E$, and $\alpha> D_{3\kappa R}^4 \norm{f}_{L^1(\mi;B)}/\mi(E) $. There are   a countable subset $N$ of $E$, a function $r\colon N\to (0,R]$, and $\mi$-measurable functions $g,h_x\colon X\to B$, $x\in N$, such that the following hold:
	\begin{enumerate}
		\item[\textnormal{(1)}] $f=g+\sum_{x\in N} h_x$;
		
		\item[\textnormal{(2)}] $g\in \chi_{B(E,R/2)}L^\infty(\mi;X)$ and
        \begin{equation*}
            \norm{g}_{L^\infty(\mi;X)}\meg D_{3\kappa R}^2\alpha;
        \end{equation*}
		
		\item[\textnormal{(3)}] $g\in L^1(\mi;B)$ and
        \begin{equation*}
            \norm{g}_{L^1(\mi;B)}\meg 3 \norm{f}_{L^1(\mi;B)};
        \end{equation*}
		
		\item[\textnormal{(4)}] $h_x$ is concentrated in $B(x,r(x))\subseteq B(E,R/2)$ for every $x\in N$, and 
		\[
		\sum_{x\in N} \mi(B(x,r(x)))\meg \frac{D_{3\kappa R}^6 \norm{f}_{L^1(\mi;B)}}{\alpha};
		\]
		
		\item[\textnormal{(5)}] $\int_X h_x\,\dd \mi=0$ for every $x\in N$;
		
		\item[\textnormal{(6)}] $h_x\in L^\infty(\mi;B)$ and
        \begin{equation*}
            \norm{h_x}_{L^\infty(\mi;B)}\meg 2 \norm{f}_{L^\infty(\mi;B)}
        \end{equation*}
        for every $x\in N$, moreover
        \begin{equation*}
            \sum_{x\in N} \norm{h_x}_{L^1(\mi;B)}\meg 2 \norm{f}_{L^1(\mi;B)};
        \end{equation*}
		
		\item[\textnormal{(7)}] the family of balls $(B(x,r(x))_{x\in N}$ has uniformly bounded overlap, in particular
        \begin{equation*}
            \sum_{x\in N} \chi_{B(x,r(x))} \meg D_{3\kappa R}^5.
        \end{equation*}
	\end{enumerate}
\end{lem}

This is a somewhat `local' Calder\'on--Zygmund decomposition. As one may expect, it only applies to functions with bounded support and may lead to suitable (but controlled) enlargements in the supports.

\begin{proof}
	Observe that, if $x\in X\setminus B(E,R/2)$ and if $y\in X$ and $r\in (0,\kappa R]$ are such that $x\in B(y,r)$ and $\dashint_{B(y,r)} \abs{f}\,\dd \mi>0$, then $r>R/4$ and $d(y,E)<r\meg \kappa R$, so that
    \begin{equation*}
        \mi(B(y,r))\Meg D_{3\kappa R}^{-4} \mi(B(y,(1+\kappa)R))\Meg D_{3\kappa R}^{-4}\mi(E),
    \end{equation*}
    since $(1+\kappa)R\meg 4R$.
    Consequently,
	\[
	\dashint_{B(y,r)} \abs{f}\,\dd \mi\meg \frac{\norm{f}_{L^1(\mi;B)}}{\mi(B(y,r))}\meg \frac{D_{3\kappa R}^4\norm{f}_{L^1(\mi;B)}}{\mi(E)} .
	\]
	We have thus proved that
    \begin{equation*}
        (\cM_{\kappa R} f)(x)\meg \frac{D_{3\kappa R}^4\norm{f}_{L^1(\mi;B)}}{\mi(E)}<\alpha
    \end{equation*}
    for every $x\in X\setminus B(E,R/2)$.
	It follows that the level set
    \begin{equation*}
        U_\alpha\coloneqq \Set{x\in X\colon (\cM_{\kappa R} f)(x)>\alpha}
    \end{equation*}
    is contained in $B(E,R/2)$.   In addition, by Lemma~\ref{lem:3b} applied to  the bounded open set $B(E,R/2)$ with $\diam(B(E,R/2))<2R$, we see that $U_\alpha $ is open and that
    \begin{equation*}
        \mi(U_\alpha)\meg \frac{D_{3\kappa R}^4 \norm{f}_{L^1(\mi;B)}}{\alpha}<\mi(E)\meg \mi(X),
    \end{equation*}
    so that $U_\alpha\neq X$. 
	Then, by Lemma~\ref{lem:2b} there are a (possibly empty) countable subset $N$ of $U_\alpha$ and a function $r\colon N\to (0,R]$ such that  $ \bigcup_{x\in N} B(x,r(x))=U_\alpha\subseteq B(E,R/2)$,  $\sum_{x\in N} \chi_{B(x,r(x))}\meg D_{3\kappa R}^5$ (whence (7)), the $B(x,r(x)/3)$, $x\in N$, are pairwise disjoint, and $B(x,\kappa r(x)) \not \subseteq U_\alpha$ for every $x\in U_\alpha$.
    
	Now, set
	\[
	\eta_x\coloneqq \frac{\chi_{B(x,r(x))}}{\sum_{y\in N}  \chi_{B(y,r(y))}}
	\]
	and
	\[
	h_x\coloneqq f \eta_x -\left(  \dashint_{B(x,r(x))} f \eta_x\,\dd \mi\right)  \chi_{B(x,r(x))}
	\]
	for every $x\in N$,  so that $\sum_{x\in N} \eta_x=\chi_{U_\alpha}$ pointwise, and
	\[
	g\coloneqq f \chi_{X\setminus U_\alpha} + \sum_{x\in N}\left(  \dashint_{B(x,r(x))} f \eta_x\,\dd \mi \right) \chi_{B(x,r(x))},
	\]
	so that (1) holds by contruction. Observe that (5) clearly holds, as well as (4), as
	\[
	\sum_{x\in N} \mi(B(x,r(x)))\meg D_{3\kappa R}^2\sum_{x\in N} \mi(B(x,r(x)/3))\meg D_{3\kappa R}^2 \mi(U_\alpha)\meg \frac{D_{3\kappa R}^6\norm{f}_{L^1(\mi;B)}}{\alpha}
	\]
	by Lemma~\ref{lem:3b}.
	Concerning (6), note that
	\[
	\sum_{x\in N} \norm{h_x}_{L^1(\mi;B)}\meg 2\sum_{x\in N} \norm{f \eta_x}_{L^1(\mi;B)}\meg 2  \norm{f}_{L^1(\mi;B)}
	\]
	since $\sum_{x\in N}\eta_x =\chi_{U_\alpha}$ (pointwise convergence); the other assertion is proved similarly. By the fact that $g=f-\sum_{x\in N} h_x$, (3) follows as well.
	Finally, notice that (2) is immediate for $\mi$-almost every $x\in X\setminus U_\alpha$, thanks to Lemma~\ref{lem:3b}. 
	Furthermore, for every $x\in N$,
	\[
	\dashint_{B(x,r(x))} \abs{f}\,\dd \mi \meg D_{3\kappa R}^2 \dashint_{B(x,\kappa r(x))}\abs{f}\,\dd \mi \meg D_{3\kappa R}^2 \alpha
	\]
	since $B(x,\kappa r(x)) \setminus U_\alpha \neq \emptyset$ and $\kappa r(x)\meg \kappa R$, so that 
    \begin{equation*}
        \dashint_{B(x,\kappa r(x))}\abs{f}\,\dd \mi \meg(\cM_{\kappa R} f)(y)\meg \alpha
    \end{equation*}
    for every $y\in B(x,\kappa r(x)) \setminus U_\alpha$. The assertion follows.
\end{proof}

\begin{deff}
	We define
	\[
	\phi(r,p)\coloneqq  \inf_{q\in (1,p)} \left( q'+\frac{1}{r/q-1} \right)^{(1/p-1/r)/(1-q/r)}
	\]
	for every $p\in (1,r]$, and we set $\phi(r,p)\coloneqq \phi(r',p')$ for every $p\in [r,\infty)$. 
\end{deff}

This definition is motivated by the following version of Marcinkiewicz's interpolation theorem.

\begin{lem}\label{lem:1}
	Let $(Y_j,\Mf_j,\nu_j)$, $j=1,2$, be two measure spaces ($\sigma$-finite or not), let $B_1,B_2$ be two Banach spaces, and let $V$ be a vector space of $\nu_1$-measurable mappings from $Y_1$ into $B_1$ such that $ f g\in V$ for every $f\in V$ and for every $g\in L^\infty(\nu_1)$. Let $T$ be a  linear operator from $V$ into the space of $\Mf_2$-measurable mappings from $Y_2$ into $B_2$.  Take $r\in (1,\infty]$, and assume that there are $A,B> 0$ such that
	\[
	\nu_2(\Set{y\in Y_2\colon \abs{(T f)(y)}>\lambda  })\meg \frac{A  \norm{f}_{L^1(\nu_1;B_1)}}{\lambda}
	\]
	for every $f\in V $ and for every $\lambda>0$, and such that
	\[
	\norm{T f}_{L^{r}(\nu_2;B_2)}\meg B \norm{f}_{L^r(\nu_1;B_1)}
	\]
	for every $f\in V $. Then, for every $p\in (1,r]$ and for every $f\in V$,
	\[
	\norm{T f}_{L^p(\nu_2;B_2)}\meg 2 \phi(r,p) A^{r'(1/p-1/r)} B^{r'/p'} \norm{f}_{L^p(\nu_1;B_1)}.
	\]
\end{lem}
 
For the proof, one may use~\cite[Theorem 1.3.2]{GrafakosClassical} (extended with the same proof to the vector-valued case) to find an estimate for the $(q,q)$ bound for $T$, and then apply Riesz--Thorin interpolation between $L^q$ and $L^r$. By inspection of the proof one may see that~\cite[Theorem 1.3.2]{GrafakosClassical} may be applied to $\chi_E V$ for every $\sigma$-finite $E\in \Mf_1$ (with the same bounds). Since every element of $L^p(\nu_1) $, $p<\infty$, is concentrated on some $\sigma$-finite element of $\Mf_1$, the assertion follows.  In a similar way, one sees that Riesz--Thorin theorem may be applied to $V$.  

We collect also the following elementary remarks which may help estimating $\phi(r,p)$.
The latter observation is relevant when addressing the uniform boundedness of the constants appearing in Theorems~\ref{teo:1} and~\ref{teo:2}, and in Corollary~\ref{cor:1b}. 

\begin{oss}\label{oss:2}
	Clearly, $\phi(\infty,p)=p'^{1/p}$ for every $p\in (1,\infty)$. In addition,
    \begin{equation*}
        \phi(r,p)\meg \left( \frac{\sqrt r +1}{\sqrt r-1}\right)^{(1/p-1/r)/(1-1/\sqrt r)}
    \end{equation*}
    for every $\sqrt r\meg p\meg r$, while
    \begin{equation*}
        \phi(r,p)\meg \left( p'+\frac{1}{r/p-1} \right)^{1/p}
    \end{equation*}
    for $1<p\meg \sqrt r$.
\end{oss}

\begin{proof}
	The first assertion is clear. The second one follows taking $q\to\sqrt r^-$ and $q\to p^-$, respectively, in the definition of $\phi(r,p)$. To motivate these choices, observe that the function $q\mapsto q'+\frac{1}{r/q-1}$ is $\Meg 1$, decreasing on $(1,\sqrt r]$ and increasing on $[\sqrt r,p)$, and that the function $q\mapsto(1/p-1/r)/(1-q/r) $ is increasing on $(1,p)$.
\end{proof}

\begin{oss}
	Take $C_1,C_2>1$. There is a constant $C_3>0$ such that $\phi(r,p)\meg C_3$ for every $r,p\in [C_1,\infty]$ such that $p\meg \min(C_2 r, r/(1-C_2/\log r)_+)$.\footnote{When $r=\infty$, interpret this condition as $p\meg \infty$.}
\end{oss}

\begin{proof}
	The case $p\meg r$  follows from Remark~\ref{oss:2}. Now, assume that $ r< p  $ (so that $r<\infty$, hence also $p<\infty$),  and choose $q\in (1,p')$  in the definition of $\phi(r',p')=\phi(r,p)$ so that $q'=c r$ for some $c>C_1$ (this is possible since $c r> C_1 r\Meg p$). Then,
    \begin{equation*}
        \phi(r,p) \meg \left( c r +\frac{1}{r'/q-1} \right)^{(1/p'-1/r')/(1-q/r')}.
    \end{equation*}
	Observe that $1/p'-1/r'=1/r-1/p$,
    \begin{align*}
        r'/q-1
        &= r' (1-1/(c r))-1\\
        &=(c r-1)/(c r-c)-1
        =(c-1)/(c r-c),
    \end{align*}
    and
    \begin{align*}
        1-q/r'
        &=1-[c r/(c r-1)] [(r-1)/r]\\
        &=1-(c r-c)/(c r -1)=(c-1)/(c r-1),
    \end{align*}
    so that
	\[
	\phi(r,p)\meg \left( cr +c \frac{r-1}{c-1} \right)^{(1/r-1/p)(c r-1)/(c-1)}.
	\]
	Now,  $1/r-1/p\meg 1/r(1-1/C_2)$ if $r\meg \ee^{C_2^2/(C_2-1)}$ (since $p\meg C_2 r$), and $1/r-1/p\meg C_2/(r\log r)$ if $r\Meg \ee^{C_2^2/(C_2-1)}$ (since $1-C_2/\log r\Meg 1/C_2>0$). The assertion follows.
\end{proof}

We are now in a position to state a result on the continuity of vector-valued Calder\'on--Zygmund operators (cf.~\cite[Theorem 1.1]{Grafakos} for the doubling case).  Since our results are of a somewhat `local' nature, we initially consider `localized' operators, and then show how these `localized' operators may be patched together to obtain a `global' result, under suitable assumptions.

Take two Banach spaces $B_1,B_2$, $R>0$, $r\in (1,\infty]$,   a $\mi$-measurable subset $E$ of $X$ with $\diam(E)<R$,   a continuous linear mapping
\begin{equation*}
    T\colon \chi_{B(E,R/2)}L^r (\mi;B_1)\to  \chi_{E}L^r (\mi;B_2), 
\end{equation*}
and  $K\colon E\times B(E,R/2)\to \Lc(B_1;B_2)$  a  $(\mi\otimes \mi)$-measurable map.  
Consider the following conditions:
\begin{enumerate}
	\item[(i)$_R$] for every   open  subset $F$ of $B(E,R/2)$, $K$ is $(\mi\otimes \mi)$-integrable on   $ (E\setminus F)\times F$, and   
	\[
	(T f)(x)=\int_F K(x,y) f(y)\,\dd \mi(y)
	\]
	for  $\mi$-almost every $x\in E\setminus F$, for every $f\in \chi_F L^\infty(\mi;B_1)$;
	
	\item[(ii)$_R$]   there exists $A_R>0$ such that   $\norm{ T f }_{L^r(\mi;B_2)}\meg A_R \norm{f}_{L^r(\mi;B_1)}$ for every $f\in \chi_{B(E,R/2)} L^r(\mi;B_1)$;

	\item[(iii)$_R$] there is a constant $C_R>0$ such that 
	\[
	\int_{E\setminus  B(y, 2d(y,y') )} \abs{K(x,y)-K(x,y')}\,\dd \mi(x)\meg C_R
	\]
	for $\mi$-almost every $y,y'\in B(E,R/2)$ with $d(y,y')<R$. 
\end{enumerate}

\begin{teo}\label{teo:1} 
	Keep the above notation, take $\kappa>2$, assume that $D_{3\kappa R}<\infty$ and that \textnormal{(i)$_{R}$}, \textnormal{(ii)$_{R}$}, and \textnormal{(iii)$_{R}$} hold.
	The following holds
	\[
	\norm{Tf}_{L^p(\mi;B_2)}\meg 16 \,\phi(r,p)(D_{3\kappa R}^9A_{R}+C_{R}) \norm{f}_{L^p(\mi;B_1)}
	\]
	for every $p\in (1,r]$ and for every $f\in \chi_E L^r(\mi;B_1)$.
\end{teo}

Notice that, even though the assumptions are imposed on functions concentrated on $B(E,R/2)$, the conclusions only hold for functions concentrated on $E$.

\begin{proof} 
	By~Lemma~\ref{lem:1}, it will suffice to show that  the operator $T_E\colon f \mapsto T ( \chi_E f)$ is of weak type $(1,1)$ with norm $\meg 8(D^9_{3 \kappa R}A_{ R}+C_{ R})$.
	Take $f\in \chi_E L^\infty(\mi;B_1)$ and $\lambda>0$. Set $\alpha\coloneqq \lambda/(2D_{3\kappa R}^{2}A_{ R})$. If $\alpha\meg D_{3\kappa R}^4 \norm{f}_{L^1(\mi;B_1)}/\mi(E)$, then
	\[
	\mi\left(\Set{x\in X\colon \abs{(T_E f)(x)}>\lambda}\right)\meg \mi(E) \meg  2D_{3\kappa R}^{6}A_{ R}\frac{\norm{f}_{L^1(\mi;B_1)}}{\lambda}.
	\]
	If, otherwise, $ \alpha> D_{3\kappa R}^4\norm{f}_{L^1(\mi;B_1)}/\mi(E)$, then take $N$, $r$, $g$, and the $h_x$ as in  Lemma~\ref{lem:4b}.
	Observe that   $h=\sum_{x\in N} h_x$ and $g $ are bounded and concentrated on $B(E,R/2)$ (cf.~(1), (2), and (4) of Lemma~\ref{lem:4b}); more precisely, $\norm{g}_{L^\infty(\mi;B_1)}\meg \lambda/(2 A_{R})$. Observe that, since $T_E f=T f= T g+ T h$,
	\[
	\begin{split}
		\mi\left(\Set{x\in X\colon \abs{(T_E f)(x)}>\lambda}\right)\meg&  \mi\left(\Set{x\in E\colon \abs{(T  g)(x)}>\lambda/2}\right)\\
        &+\mi\left(\Set{x\in E\colon \abs{(T  h)(x)}>\lambda/2}\right)  .
	\end{split}
	\]
	If $r<\infty$, then
	\[
	\begin{split}
		\mi\left(\Set{x\in E\colon \abs{(T  g)(x)}>\lambda/2}\right) &\meg (\lambda/2)^{-r} \norm{ T g }_{L^r( \mi;B_2)}^r\\
		&\meg \left(\frac{2 A_{R}}{\lambda}\right)^r \norm{g}_{L^r(\mi;B_1)}^r\\
		&\meg \frac{2 A_{R}}{\lambda}  \norm{g}_{L^1(\mi;B_1)}\\
		&\meg \frac{6  A_{R}}{\lambda} \norm{f}_{L^1(\mi;B_1)},
	\end{split}
	\]
	thanks to the previous remarks and (3) of Lemma~\ref{lem:4b}.
	If $r=\infty$, then $\norm{ Tg}_{L^\infty(\mi;B_2)}\meg A_{R} \norm{g}_{L^\infty(\mi;B_1)}\meg  \lambda/2 $, so that
	\[
	\begin{split}
		\mi\left(\Set{x\in E\colon \abs{(T g)(x)}>\lambda/2}\right) =0.
	\end{split}
	\]
	Next, set $B'\coloneqq \bigcup_{x\in N} B(x,2r(x))$ and notice that
	\[
	\mi\left(\Set{x\in E\colon \abs{(T  h)(x)}>\lambda/2}\right)\meg   \mi(B')+ \mi\left(\Set{x\in E\setminus B'\colon \abs{(T h)(x)}>\lambda/2}\right) 
	\]
	and that
	\begin{align*}
	    \mi(B')\meg& \sum_{x\in N} \mi(B(x, 2r(x)))\\
        \meg& D_{3\kappa R}\sum_{x\in N} \mi(B(x, r(x)))\\
        \meg& 2A_{R} D_{3\kappa R}^{9} \norm{f}_{L^1(\mi;B)}/\lambda
	\end{align*}
	thanks to (4) of Lemma~\ref{lem:4b}.
	Finally, for  $\mi$-almost every $x\in E\setminus B'$,  
	\[
	\begin{split}
		\abs{T h(x)} &=\abs*{ \int_{B'} K(x,y) h(y)\,\dd \mi(y) }\\
		&=\abs*{ \sum_{x'\in N}\int_{B'} K(x,y) h_{x'}(y)\,\dd \mi(y) }\\
		&\meg \sum_{x'\in N} \abs*{\int_{B'} K(x,y) h_{x'}(y)\,\dd \mi(y)},
	\end{split}
	\]
	where the first equality follows from (i)$_{R}$, since $h\in \chi_{B'}L^\infty(\mi;B_1)$,  whereas the second equality follows from (i)$_{R}$ and the dominated convergence theorem, since $K(x,\,\cdot\,)$ is $\mi$-integrable on $B'$ for $\mi$-almost every $x\in E\setminus B'$  and $\sum_{x'\in N} \abs{h_{x'}(y)}\meg 2D_{3\kappa R}^5 \norm{f}_{L^\infty(\mi;B_1)}$ for $\mi$-almost every $y\in B'$ by (6) and (7) of Lemma~\ref{lem:4b}.
	In addition, since $h_{x'}$ is concentrated on $B(x', r(x'))$ and  $\int_X h_{x'}\,\dd \mi=0$ by (4) and (5) of Lemma~\ref{lem:4b},
	\[
	\int_{B'} K(x,y) h_{x'}(y)\,\dd \mi(y)= \int_{B(x', r(x'))} [K(x,y)-K(x,x')] h_{x'}(y)\,\dd \mi(y)
	\]
	for every $x'\in N$ and for   $\mi$-almost every $x\in X$. Thus, by Tonelli's theorem (which may be applied since $\mi$ is finite on $B(E,R/2)\supseteq E,B'$, thanks to Lemma~\ref{lem:0}), and by (6) of Lemma~\ref{lem:4b},
	\[
	\begin{split}
		&\mi\left(\Set{x\in E\setminus B'\colon \abs{(T h)(x)}>\lambda/2}\right) \meg \frac{2}{\lambda} \int_{E\setminus B'} \abs{T h}\,\dd \mi\\
		&\qquad \meg  \frac{2}{\lambda} \sum_{x'\in N} \int_{B(x',r(x'))} \int_{E\setminus B(x',2 r(x'))} \abs*{K(x,y)-K(x,x')} \,\dd \mi(x) \abs{h_{x'}(y)}\,\dd \mi(y)\\
		&\qquad\meg \frac{2C_{R}}{\lambda} \sum_{x'\in N} \int_X  \abs{h_{x'}}\,\dd \mi\\
		&\qquad \meg \frac{4C_{R}}{\lambda}\norm{f}_{L^1(\mi;B_1)}.
	\end{split}
	\]
	The assertion follows.
\end{proof}

Now, take $R>0$, $r\in [1,\infty)$, a $\mi$-measurable subset $E$ of $X$ with $\diam(E)<R$, a continuous linear mapping 
\begin{equation*}
    T\colon \chi_E L^r (\mi;B_1)\to  \chi_{B(E,R/2)}L^r (\mi;B_2),
\end{equation*}
and $K\colon B(E,R/2)\times E\to \Lc(B_1;B_2)$ a  $(\mi\otimes \mi)$-measurable map. 
 Consider the following conditions:
 \begin{enumerate}
 	\item[(i)$'_R$] for every  open  subset $F$ of $B(E,R/2)$, $K$ is $(\mi\otimes \mi)$-integrable on   $ F\times (E\setminus F)$, and   
 	\[
 	(T f)(x)=\int_{E\setminus F} K(x,y) f(y)\,\dd \mi(y)
 	\]
 	for $\mi$-almost every $x\in F$ and for every $f\in~\chi_{E\setminus  F} L^\infty(\mi;B_1)$;
 	
 	\item[(ii)$'_R$]   there exists $A_R>0$ such that   $\norm{ T f }_{L^r(\mi;B_2)}\meg A_R \norm{f}_{L^r(\mi;B_1)}$ for every $f\in\chi_{E} L^r(\mi;B_1)$;
 	
 	\item[(iii)$'_R$] there is a constant $C_R>0$ such that
 	\[
 	\int_{E\setminus  B(x, 2d(x,x') )} \abs{K(x,y)-K(x',y)}\,\dd \mi(y)\meg C_R
 	\]
 	for $\mi$-almost  every $x,x'\in B(E,R/2)$  with $d(x,x')<R$.
 \end{enumerate}

\begin{cor}\label{cor:1b}
	Keep the above notation, take $\kappa>2$, assume that $D_{3\kappa R}<\infty$ and that  
	\textnormal{(i)$'_{R}$}, \textnormal{(ii)$'_{R}$}, and \textnormal{(iii)$'_{R}$}  hold.
	The following holds,
	\[
	\norm{\chi_E T f}_{L^p(\mi;B_2)}\meg 16\, \phi(r,p) (D_{3\kappa R}^9A_{R}+C_{R}) \norm{f}_{L^p(\mi;B_1)}
	\]
	for every $p\in [r,\infty)$  and for every $f\in \chi_E L^\infty(\mi;B_1)$.
\end{cor}
 
Notice that, in comparison with~\cite{Grafakos}, we have to proceed in a slightly different way, since it is in general false that the dual of $L^r(\mi;B_1)$ is $L^{r'}(\mi;B_1')$, as implicitly claimed in the cited reference. Fortunately, it suffices to prove the result replacing $B_1$ with its finite-dimensional vector subspaces to solve this issue. 
\begin{proof}
	Let $\Ff$ be the set of  finite-dimensional vector subspaces of $B_1$, and take $F\in \Ff$. Denote with $T_F$ the restriction of $T$ to $\chi_{E}L^r(\mi;F)$. Observe that  $\chi_{B(E,R/2)}L^{r'}(\mi;B_2')$ embeds canonically into the dual of $\chi_{B(E,R/2)}L^r(\mi;B_2)$, thanks to~\cite[Proposition 3 of Chapter IV, \S\ 6, No.\ 4]{BourbakiInt1}.  
	Therefore, the transpose of  the continuous linear mapping $T_F\colon \chi_{E}L^r(\mi;F)\to \chi_{B(E,R/2)}L^r(\mi;B_2)$  
	induces a   continuous  linear mapping $T'_F$ of $ \chi_{B(E,R/2)}L^{r'}(\mi;B_2')$ into  $ \chi_{E}L^{r'}(\mi;F')$. 
	
	Observe that, by (ii)$'_{R}$,   for every $f\in \chi_{B(E,R/2)} L^{r'}(\mi;B_2')$,
	\[
	\begin{split}
		\sup_{  \norm{g}_{L^r(\mi;F)}\meg 1 } \abs{\langle \chi_{E} g ,   T'_F f \rangle}=\sup_{  \norm{g}_{L^r(\mi;F)}\meg 1}  \abs{\langle T(\chi_E g) ,  f \rangle}\meg A_{R} \norm{f}_{L^{r'}(\mi;B_2')}
	\end{split}
	\]
	thanks to~\cite[Proposition 3 of Chapter IV, \S\ 6, No.\ 4]{BourbakiInt1}.
	In particular,    the analogue of (ii)$_{R}$,   with $T$ and $r$  replaced by $ T_F'$ and $r'$, respectively,   holds. 
	Next, denote with $\trasp K_F(y,x)$ the transpose of the composite of the canonical (isometric) inclusion $F\to B_1$ with $K(x,y)$, so that $\trasp K_F\colon E\times B(E,R/2)\to \Lc(B_2',F')$ is $(\mi \otimes \mi)$-measurable. Take an open  subset $E'$ of $B(E,R/2)$, $f\in \chi_{E'} L^\infty(\mi;B_2')$ and   $g\in \chi_{E\setminus E' }L^{\infty}(\mi;F )$, and observe that 
	\[
	\begin{split}
		\langle g ,  T'_F f\rangle&=\langle T_F g, f \rangle\\
			&=\int_{E'} \bigg\langle \int_{E\setminus E'} K(x,y) g (y)\,\dd \mi(y), f(x)\bigg\rangle\,\dd \mi(x)\\
			&=\int_{E'} \int_{E\setminus E'} \bigg\langle g (y), \trasp K(y,x) f(x)\bigg\rangle\,\dd \mi(y)\,\dd \mi(x)\\
			&=\int_{ E\setminus E'} \bigg\langle   g (y), \int_{E'} \trasp K_F(y,x) f(x)\,\dd \mi(x)\bigg\rangle\,\dd  \mi(y)
	\end{split}
	\]
	by Fubini's theorem.  Consequently,
	\[
	(T'_F f)(y)=\int_{E'} \trasp K_F(y,x) f(x)\,\dd \mi(x)
	\]
	for $\mi$-almost every $y\in E\setminus E'$, so that the analogue of (i)$_{R}$, with $T$ and $K$  replaced by $  T'_{F}$ and $\trasp K_F$, respectively, holds. In addition, it is clear that the analogue  of   (iii)$_{R}$, with   $K$  and $C_{R}$ replaced by   $\trasp K_F$  and $C_{R}'$, respectively, also holds true.
	Thus, Theorem~\ref{teo:1} shows that
	\[
	\norm{T'_{F} f}_{L^p(\mi;F')}\meg 16\, \phi(r,p) (D_{3\kappa R}^9 A_{R}+C_{R}) \norm{f}_{L^p(\mi;B_2')}
	\]
	for every $p\in (1,r']$, for every $f\in \chi_{E}  L^\infty(\mi;B_2')$, and for every $F\in \Ff$. Then, for every $p\in [r,\infty)$, for every $f\in \chi_E L^\infty(\mi;F')$, and for every $g\in~\chi_E L^{\infty}(\mi; B_2')$,
	\begin{align*}
	    \abs{\langle T f, g\rangle}
        &=\abs{\langle T_{F} f,g \rangle}
        =\abs{\langle f,  T'_{ F} g\rangle}\\
        &\meg 16\, \phi(r,p)  (D_{3\kappa R}^9 A_{R}+C_{R}) \norm{f}_{L^p(\mi;B_1)} \norm{g}_{L^{p'}(\mi;B_2')}.
	\end{align*}
	Using the fact that $\sum_{F\in \Ff} \chi_E L^\infty(\mi;F)$ is dense in $\chi_E L^\infty(\mi;B_1)$ and~\cite[Proposition 3 of Chapter IV, \S\ 6, No.\ 4]{BourbakiInt1}, we then deduce that
	\[
	\norm{\chi_E T f}_{L^p(\mi;B_2)}\meg 16\, \phi(r,p)  (D_{3\kappa R}^9 A_{R}+C_{R}) \norm{f}_{L^p(\mi;B_1)} 
	\]
	for every $f\in \chi_E L^\infty(\mi;B_1)$. 
\end{proof}

We now show how the previous results may be `patched together' in order to get a `global' result.

\begin{deff}
	Define $L^\infty_{c}(\mi;B_1)$ as the space of elements of $L^\infty(\mi;B_1)$ with compact support. 
\end{deff}

If $D_R<\infty$, then $L^\infty_{c}(\mi;B_1)$ may be characterized as the space of elements of $L^\infty(\mi;B_1)$ which are concentrated in a finite union of $\mi$-measurable sets with diameter $<2R$ (since these sets are necessarily precompact).

\begin{cor}\label{cor:2}
	Take $\kappa>2$, $r\in (1,\infty]$ (resp.\ $r\in [1,\infty)$) and $R>0$, and assume that $D_{3 \kappa R}<\infty$.
	Let $T$ be a linear operator from $L^\infty_{c}(\mi;B_1)$ into the space of $\mi$-measurable functions from $X$ into $B_2$, and let $K\colon X\times X\to \Lc(B_1;B_2)$ be a a $(\mi\otimes \mi)$-measurable mapping such that the following hold:
	\begin{enumerate}
		\item[\textnormal{(a)}] for every $\mi$-measurable subset $E$ of $X$ with $\diam(E)<R$, the mapping $f\mapsto \chi_E T(\chi_{B(E,R/2)} f)$ and the restriction of $K$ to $E\times B(E,R/2)$ satisfy (i)$_R$, (ii)$_R$, and (iii)$_R$ (resp.\ the mapping $f\mapsto \chi_{B(E,R/2)}T(\chi_{E} f)$ and the restriction of $K$ to $B(E,R/2)\times E$ satisfy (i)$'_R$, (ii)$'_R$, and (iii)$'_R$);
		
		\item[\textnormal{(b)}] $(T f)(x)=T(\chi_{\overline B(E,R/3)}f)(x)$ for  $\mi$-almost every $x\in E$, for every $f\in L^\infty_c(\mi;B_1)$, and for every $\mi$-measurable subset $E$ of $X$ with $\diam(E)<R/4$.
	\end{enumerate}
	Then, 
	\[
	\norm{T f}_{L^p(\mi;B_2)}\meg 16\, \phi(r,p)D_R^{5/p}(D_{3\kappa R}^9A_{R}+C_{R}) \norm{f}_{L^p(\mi;B_1)}
	\]
	for every $p\in (1,r]$ (resp.\ for every $p\in [r,\infty)$) and for every $f\in L^\infty_{c}(\mi;B_1)$.
\end{cor}

Notice that condition (b) holds, for example, when $T$ is `globally' defined by the integral kernel $K$, and $\supp K\subseteq \Set{(x,y)\in X\times X\colon d(x,y)\meg R/3}$.

\begin{proof}
	Let $(x_j)_{j\in J}$ be a maximal $(R/8)$-separated family of elements of $X$, so that the $B(x_j, R/16)$ are pairwise disjoint, while the $B(x_j, R/8)$ cover $X$. Then, applying Theorem~\ref{teo:1} (resp.\ Corollary~\ref{cor:1b}) to the   set $E=B(x_j,R(1/3+1/8))$, we see that
	\[
	\begin{split}
		\norm{\chi_{B(x_j,R/8)} T f}_{L^p(\mi;B_2)}&=\norm{\chi_{B(x_j,R/8)} T (\chi_{B(x_j,R(1/3+1/8))}f)}_{L^p(\mi;B_2)}\\
		&\meg \norm{\chi_{B(x_j,R(1/3+1/8))} T (\chi_{B(x_j,R(1/3+1/8))}f)}_{L^p(\mi;B_2)}\\
		&\meg 16\, \phi(r,p)(D_{3\kappa R}^9A_{R}+C_{R}) \norm{\chi_{B(x_j,R/2)}f}_{L^p(\mi;B_1)}
	\end{split}
	\]
	for every $j\in J$. If $p=\infty$, this leads to the conclusion since $\mi$ is inner regular. If, otherwise, $p<\infty$, then summing up all these inequalities\footnote{Cf.~\cite[Proposition 8 of Chapter IV, \S\ 1, No.\ 2]{BourbakiInt1} for the case in which $J$ is uncountable. Essentially, one should observe that $\chi_{B(x_j,R/2)}T f$ does not vanish $\mi$-almost everywhere only for  countably many $j$ and use the inner regularity of $\mi$.} raised to the power $p$, and taking into account the fact that $\sum_{j\in J} \chi_{B(x_j,R/2)}\meg D_R^5$ (which may be proved essentially as in the proof of Lemma~\ref{lem:2b}), we get to the conclusion.
\end{proof}

We may now proceed to show the vector-valued boundedness of the maximal function $\cM_R$. For technical reasons, we prefer to state it for the \emph{centred} Hardy--Littlewood maximal function $\widetilde \cM_{2R}$ instead.

\begin{teo}\label{teo:2} 
	Take $R>0$ such that $D_{R}<\infty$, $k\in\N $, $p_1,\dots, p_{k+1}\in (1,\infty)$, and $\sigma$-finite measure spaces $(Y_j,\Mf_j,\nu_j)$, $j=1,\dots, k$.
	Then, 
	\begin{align*}
	    &\norm*{[\widetilde\cM_{2R}f(y,\,\cdot\,)](x)}_{L^{p_1,\dots,p_{k+1}}_{(y,x)}(\nu_1,\dots, \nu_k,\mi)}\\
        &\qquad\meg (1+16^{k+1}) D_{R}^{25+16k}  p_1'^{1/p_1} \prod_{j=1}^{k} \phi(p_j,p_{j+1})   \norm{f}_{L^{p_1,\dots,p_{k+1}} (\nu_1,\dots, \nu_k,\mi)}.
	\end{align*}
	for every $f\in L^{p_1,\dots,  p_{k+1}}(\nu_1,\dots, \nu_k,\mi)$.
\end{teo}

We define $L^{p_1,\cdots,p_{k+1}}(\nu_1,\dots, \nu_k,\mi)$ as the space of $(\nu_1\otimes\cdots \otimes \nu_k\otimes \mi)$-measurable functions $f\colon Y_1\times \cdots \times Y_k\times X\to \C$ such that\footnote{We observe explicitly the $\mi$ is the only measure which is not required to be $\sigma$-finite. This is due to the fact that it is the `last' one and also inner regular: measurability is then `local' and the integrals on $X$ are the limits of the integrals on the compact subsets of $X$ in a suitable sense.}
\[
\left(\int_X \left( \int_{Y_k}\cdots \left(\int_{Y_1} \abs{f(y_1,\dots, y_k,x)}^{p_1}\dd \nu_1(y_1)\right)^{p_2/p_1}  \cdots  \dd \nu_k(y_k) \right)^{p_{k+1}/p_k} \dd \mi(x) \right)^{1/p_{k+1}} 
\]
is finite.
We remark explicitly that in the statement we wrote
\begin{equation*}
    \norm*{\big(\widetilde\cM_{2R} f(y,\,\cdot\,)\big)(x)}_{L^{p_1,\dots,p_{k+1}}_{(y,x)}(\nu_1,\dots, \nu_k,\mi)}
\end{equation*}
instead of
\begin{equation*}
    \norm*{ (y_1,\dots,y_k,x)\mapsto \big(\widetilde\cM_{2R} f(y_1,\dots,y_k,\,\cdot\,)\big)(x)  }_{L^{p_1,\dots,p_{k+1}} (\nu_1,\dots, \nu_k,\mi)}
\end{equation*}
for simplicity.
We recall a basic result which will be used repeatedly in the proof; we consider the case $k=1$ for simplicity.

\begin{lem}
	Take $p,q\in [1,\infty)$, a Banach space $B$, and a $\sigma$-finite measure space $(Y,\Mf,\nu)$. Then, the mapping
	\[
	L^{p,q}(\nu,\mi;B)\ni f\mapsto [x\mapsto f(\,\cdot\,,x)]\in L^q(\mi;L^p(\nu;B))
	\]
	is an isometric isomorphism.
\end{lem}

The main point of the proof is showing that the mapping $X\ni x\mapsto f(\,\cdot\,,x)\in L^p(\nu;B)$ is $\mi$-measurable for every $f\in L^{p,q}(\nu,\mi;B)$; this is proved reducing to integrable simple functions, and then to characteristic functions of the form $\chi_N$, for $N$ $(\nu\otimes \mi)$-negligible (in which case the assertion follows from Tonelli's theorem), or $\chi_{A\times B}$ (in which case the assertion is obvious). The fact that the mapping is isometric is then obvious, while the fact that it is onto follows from the elementary observation that its image contains all simple functions.

Let us now discuss our choice to consider $\widetilde \cM_R$ instead of $\cM_R$: this is due to the fact that we were not able to show that the function   $(y,x)\mapsto [\cM_R f(y,\,\cdot\,)](x)$ is  $(\nu_1\otimes \cdots \otimes\nu_k\otimes \mi)$-measurable. For the centred maximal function, measurability is a consequence of the following elementary remark.

\begin{oss}\label{oss:3}
	Let $(Y,\Mf,\nu)$ be a $\sigma$-finite measure space and take $R>0$ and a $(\nu\otimes \mi)$-measurable function $f\colon Y\times X\to \C$. Assume that $\mi(B(x,R))<\infty$ for every $x\in X$. Then, the function
	\[
	Y\times X \ni (y,x)\mapsto \big(\widetilde \cM_R f(y,\,\cdot\,)\big)(x)\in [0,+\infty]
	\]
	is $(\nu\otimes \mi)$-measurable.
\end{oss}

One should apply this result with $(Y,\Mf,\nu)$ chosen as the (completed) product of the measure spaces $(Y_j,\Mf_j,\nu_j)$, $j=1,\dots, k$, with the notation of Theorem~\ref{teo:2}.  
Besides that, one has to interpret $(\nu\otimes \mi)$-measurability in an appropriate way, so that a subset $E$ of $Y\times X$ is $(\nu\otimes \mi)$-measurable if and only if, for every compact subset $K$ of $X$,  $E\cap (Y\times K)$ belongs to the completed  product of the $\sigma$-algebra  $\Mf$ with the Borel $\sigma$-algebra on $X$. One may then verify that $\nu\otimes \mi$ induces a measure on the $\sigma$-algebra of $(\nu\otimes \mi)$-measurable sets.

\begin{proof}
	Notice first that, for every $r\in (0,R]$, the mapping $(x,x')\mapsto \chi_{D_r}(x,x')$, where $D_r\coloneqq\{\,(x,x')\in X\times X\colon d(x,x')<r\,\}$, is lower semi-continuous on $X\times X$. Therefore, by Tonelli's theorem, the mapping $x\mapsto \mi(B(x,r))= \int_X \chi_{D_r}(x,x')\,\dd \mi(x')$ is lower semi-continuous as well. Again by Tonelli's theorem, since the mapping $(y,x,x')\mapsto f(y,x')\chi_{D_r}(x,x')$ is $(\nu\otimes \mi\otimes \mi)$-measurable, the mapping $(y,x)\mapsto \int_{B(x,r)} \abs{f(y,x')}\,\dd \mi(x')$ is $(\nu\otimes \mi)$-measurable.\footnote{Here one needs to pay some attention if $\mi$ is not $\sigma$-finite. However,  we may reduce to the case $r<R$ and to proving measurability on $Y\times B(x_0,R-r)$ for every $x_0\in X$; since $(y,x,x')\mapsto \chi_{B(x_0,R-r)}(x) f(y,x')\chi_{D_r}(x,x')$ is concentrated on  $Y\times B(x_0,R)\times B(x_0,R-r)$, one may apply the usual Tonelli's theorem.} Consequently, the mapping
	\[
	Y\times X \ni (y,x)\mapsto \sup_{r\in (0,R]\cap \Q} \dashint_{B(x,r)} \abs{f(y,x')}\,\dd \mi(x')\in [0,+\infty]
	\]
	is $(\nu\otimes \mi)$-measurable. The conclusion follows from the fact that the mapping 
	\[
	(0,R]\ni r\mapsto \dashint_{B(x,r)} \abs{f(y,x')}\,\dd \mi(x')\in [0,+\infty]
	\]
	is left continuous by monotone convergence, so that
	\[
	\sup_{r\in (0,R]\cap \Q} \dashint_{B(x,r)} \abs{f(y,x')}\,\dd \mi(x')=\sup_{r\in (0,R] } \dashint_{B(x,r)} \abs{f(y,x')}\,\dd \mi(x')
	\]
	for every $x\in X$.
\end{proof}

We now prove a technical lemma, following the approach in the doubling case. 

\begin{lem}\label{lem:5b}
	Take $R>0$ so that $D_{4R}<\infty$. 
	Define $V_r(x)\coloneqq \mi(B(x,r))$ for every $x\in X$ and for every $r>0$.
	There are a family $(S_r)_{r\in (0,R]}$ of positive continuous functions on $X\times X$ and a constant $C>1$ such that the following hold:
	\begin{enumerate}
		\item[\textnormal{(1)}] $S_r(x,y)=0$ if $d(x,y)\Meg 3 r$;
		
		\item[\textnormal{(2)}] $S_r(x,y)=S_r(y,x)$ for every $x,y\in X$;
		
		\item[\textnormal{(3)}] $\int_X S_r(x,y)\,\dd \mi(y)=1$ for every $x\in X$;
		
		\item[\textnormal{(4)}] for every $x,y\in X$ with $d(x,y)<r/2$
        \begin{equation*}
            S_r(x,y)\Meg \frac{1}{  D_{4R}^5 \min(V_r(x),V_r(y))};
        \end{equation*}
		
		\item[\textnormal{(5)}] for every $x,y\in X$
        \begin{equation*}
            S_r(x,y)\meg \frac{2D_{4R}^2}{V_r(x)+V_r(y)};
        \end{equation*}
		
		\item[\textnormal{(6)}] for every $x,x',y\in X$
        \begin{equation*}
            \abs{S_r(x,y)-S_r(x',y)}\meg \frac{C D_{4R}^{8} d(x,x')}{r(\min(V_r(x),V_r(x'))+V_r(y))}.
        \end{equation*}
	\end{enumerate}
\end{lem}

This lemma is essentially a particular case of~\cite[Lemma 2.1]{Muller}. We (slightly) simplify some arguments and highlight the dependence on the appropriate doubling constant.

\begin{proof}
	Take $h\in C^1(\R)$ with $\chi_{[0,1]}\meg h \meg \chi_{[-1,\eta]}$, for some $\eta\in (1,7/6]$. For every $r\in (0,R]$, consider the operator $T_r$ defined so that
	\[
	(T_r f)(x) =\int_X h(d(x,y)/r) f(y)\,\dd \mi(y)
	\]
	for every $\mi$-measurable function $f$ and for every $x\in X$ for which the above integral is defined. Observe that $T_r$ maps $L^\infty(\mi)$ into $C(X)$ continuously by Lemma~\ref{lem:0} and by dominated convergence, that $V_r\meg T_r 1 \meg V_{\eta r}$, and that
    \begin{align*}
        D_{4R}^{-2} V_{2\eta r}(x)
        \meg& D_{4R}^{-2}V_{3\eta r}(y)
        \meg V_r(y)
        \meg (T_r 1)(y)\\
        \meg& V_{\eta r}(y)
        \meg V_{2\eta r}(x)
        \meg D_{4R}^2 V_r(x)
    \end{align*}
	for every $ x,y\in X$ such that $d(x,y)\meg \eta r$. Consequently, for every $x\in X$,
	\begin{align*}
	    D_{4R}^{-2}
        \meg&\int_{B(x,r) } \frac{1}{(T_r 1)(y)} \,\dd \mi(y)  \meg T_r(1/T_r 1)(x)\\
        \meg& \int_{B(x, \eta r)} \frac{1}{(T_r 1)(y)} \,\dd \mi(y)
        \meg D_{4R}^{2}.
	\end{align*}
	Define, for every $x,y\in X$,
	\[
	S_r(x,y)\coloneqq \frac{1}{(T_r1)(x) (T_r 1)(y) } \int_X\frac{ h(d(x,z)/r) h(d(z,y)/r)}{T_r(1/T_r 1)(z)}\,\dd \mi(z).
	\]
	The continuity of $S_r$ is clear, as well as (1) (with $3$ replaced by $2\eta$) and (2); (3) follows from Fubini's theorem. 
	Concerning (4), observe that, if $d(x,y)<r/2$, then
	\begin{align*}
	    S_r(x,y)
        \Meg& \frac{1}{D_{4R}^{2}V_{\eta r}(x) V_{\eta r}(y)} \mi(B(x,r)\cap B(y,r))\\
        \Meg& \frac{1}{D_{4R}^{4}V_{ r}(x) V_{ r}(y)} \mi(B(x,r)\cap B(y,r)),
	\end{align*}
	whence our assertion since
    \begin{equation*}
        \mi(B(x,r)\cap B(y,r))\Meg \max(V_{r/2}(x), V_{r/2}(y))\Meg D_{4R}^{-1} \max(V_r(x), V_r(y)).
    \end{equation*}
	Now, note that
	\[
	S_r(x,y)\meg  \frac{D_{4R}^{2} }{(T_r1)(x)   (T_r 1)(y) } \int_X  h(d(z,y)/r )\,\dd \mi(z)\meg  \frac{D_{4R}^{2} }{V_r(x)}.
	\]
	By symmetry, we also have $S_r(x,y)\meg \frac{D_{4R}^{2} }{V_r(y)}$, which gives (5).
	
	Finally, let us prove (6). Notice that, if $d(x,x')\Meg 4\eta r$, then either $S_r(x,y)=0$ or $S_r(x',y)=0$, so that the assertion follows from (5). Then, assume that   $d(x,x')<4\eta r$ and observe that, if $S_r(x,y)-S_r(x',y)\neq 0$, then  $d(x,y)<6\eta r$ and $d(x',y)<6\eta r$, so that arguing as before we can see that
	\[
	\min(V_r(x),V_r(x')) \Meg D_{4R}^{-3} \min(V_{8 r}(x),V_{8r}(x'))\Meg \frac{\min(V_r(x),V_r(x'))+V_r(y)}{2 D_{4R}^3}.
	\]
	Next, observe that
	\[
	\begin{split}
		&S_r(x,y)-S_r(x',y)\\
        &= \left( \frac{1}{(T_r1)(x)}-\frac{1}{(T_r1)(x')}\right) \frac{1}{T_r(y) }\int_X\frac{ h(d(x,z)/r) h(d(z,y)/r)}{T_r(1/T_r 1)(z)}\,\dd \mi(z)\\
		&\qquad + \frac{1}{(T_r1)(x') (T_r1)(y)}\int_X\frac{ [h(d(x,z)/r)-h(d(x',z)/r)] h(d(z,y)/r)}{T_r(1/T_r 1)(z)}\,\dd \mi(z).
	\end{split}
	\]
	Concerning the first term,
	\begin{align*}
		\abs*{\frac{1}{(T_r1)(x)}-\frac{1}{(T_r1)(x')}}&\meg \frac{\abs{(T_r1)(x)-(T_r1)(x')}}{V_r(x) V_r(x')}\\
		&\meg \frac{1}{V_r(x) V_r(x')} \norm{h'}_{L^\infty(\R)} \int_{B(x',5\eta r)} \frac{\abs{d(x,y)-d(x',y)}}{r}\,\dd \mi(y)\\
		&\meg \frac{V_{8 r}(x')}{V_r(x) V_r(x')} \norm{h'}_{L^\infty(\R)}\frac{d(x,x')}{r}\\
		&\meg D_{4R}^3  \norm{h'}_{L^\infty(\R)}\frac{d(x,x')}{r V_r(x)},
	\end{align*}
	so that 
	\begin{align*}
	&\abs*{ \left( \frac{1}{(T_r1)(x)}-\frac{1}{(T_r1)(x')}\right) \frac{1}{(T_r1)(y) }\int_X\frac{ h(d(x,z)/r) h(d(z,y)/r)}{T_r(1/T_r 1)(z)}\,\dd \mi(z)}\\
    &\qquad\meg D_{4R}^5  \norm{h'}_{L^\infty(\R)}\frac{d(x,x')}{r V_r(x)}.
	\end{align*}
	Concerning the second term,   arguing as before we see that
	\[
	\begin{split}
		\frac{1}{(T_r1)(x') (T_r1)(y)}&\int_X\frac{ \abs{h(d(x,z)/r)-h(d(x',z)/r)} h(d(z,y)/r)}{T_r(1/T_r 1)(z)}\,\dd \mi(z)\\
        &\meg \frac{D_{4R}^{2}}{(T_r1)(x')}   \norm{h'}_{L^\infty(\R)} \frac{d(x,x')}{r}  \\
		&\meg \frac{D_{4R}^{2}}{V_r(x')}   \norm{h'}_{L^\infty(\R)} \frac{d(x,x')}{r}  ,
	\end{split}
	\]
	so that the assertion follows. 
\end{proof}
 
\begin{proof}[Proof of Theorem~\ref{teo:2}.] 
 	Take $(S_r)_{r\in (0,  R/4]}$ and define $V_r$ as in Lemma~\ref{lem:5b}, so that there is a  constant $C_1>1$ such that
 	\begin{align*}
 	    \frac{1}{ D_{R}^6 V_{r/2(x)}}\chi_{B(x,r/2)}(y)
        &\meg \frac{1}{  D_{R}^5 V_r(x)}\chi_{B(x,r/2)}(y)
        \meg S_r(x,y)\\
        &\meg  \frac{2D_{R}^2 }{ V_r(x)} \chi_{B(x,3r)}(y)
        \meg \frac{2D_{R}^4 }{ V_{3r}(x)} \chi_{B(x,3r)}(y)
 	\end{align*}
	for every $x,y\in X$ and for every $r\in (0,  R/4]$, while
	\[
	\abs{S_r(x,y)-S_r(x,y')}\meg \frac{C_1 D_{R}^{8} d(y,y')}{r \min(V_r(y),V_r(y'))}
	\]
	for every $x,y,y'\in X$. Set $R'\coloneqq R/63$.
	For every finite  subset $J$ of $\N$, consider the linear operator $T_J\colon L^\infty(\mi)\to L^\infty(\mi;\ell^\infty(J))$ defined so that $T_J f(x)= \int_X K_J(x,y)f(y)\,\dd \mi(y)$, where
	\[
	K_J(x,y)=( S_{ 2^{-j}  R'}(x,y)   )_{j\in J}
	\]
	for every $x,y\in X$. 	Observe that $K_J\colon X\times X\to \ell^\infty(J)$ is continuous, so that  $T_J$ is   well defined. 
	In addition, clearly
	\[
	\norm{(T_J f)(x)}_{\ell^\infty(J)}\meg 2D_{4R'}^4\norm{f}_{L^\infty(\mi)}
	\]
	for every $J$, for every $x\in X$, and for every $f\in L^\infty(\mi)$.
	Now, take $y,y'\in X$, with $y\neq y'$. Then,  
	\[
	\begin{split}
	&\int_{X\setminus B(y,2d(y,y'))} \norm{K_J(x,y)-K_J(x,y')}_{\ell^\infty(J)}\,\dd \mi(x)\\
    &\quad\meg \sum_{j\in \N} \int_{ X\setminus B(y,2d(y,y'))} \abs{S_{2^{-j}  R'}(x,y)-S_{2^{-j}  R'}(x,y')}\,\dd \mi(x)\\
		&\quad=  \sum_{2^{-j}3  R'\Meg d(y,y')} \int_{ (B(y,2^{-j}3  R')\cup B(y',2^{-j}3  R'))\setminus B(y,2d(y,y'))} \abs{S_{2^{-j}  R'}(x,y)-S_{2^{-j}  R'}(x,y')}\,\dd \mi(x)\\
		&\quad\meg \frac{C_1 D_{R}^{8}}{   R'}\sum_{2^{-j}3  R'\Meg d(y,y')} \int_{  B(y,2^{-j}3   R')\cup B(y',2^{-j}3  R')} \frac{2^j d(y,y')}{\min(V_{2^{-j}  R'}(y),V_{2^{-j}  R'}(y'))}\,\dd \mi(x)\\
		&\quad\meg  \frac{C_1D_{R}^{8}}{   R'}\sum_{2^{-j}3  R'\Meg d(y,y')}   2^j d(y,y') \frac{ V_{2^{-j}3  R'}(y)+V_{2^{-j}3  R'}(y')}{\min(V_{2^{-j}  R'}(y),V_{2^{-j}  R'}(y'))}\\
			&\quad \meg 2\frac{C_1D_{ R}^{8}}{   R'}\sum_{2^{-j}3 R'\Meg d(y,y')} 2^j d(y,y')   \frac{\min(V_{2^{3-j}  R'}(y),V_{2^{3-j}  R'}(y'))}{\min(V_{2^{-j}  R'}(y),V_{2^{-j}  R'}(y'))}\\
			&\quad \meg 12C_1 D_{R}^{11} .
	\end{split}
	\] 
	By means of Corollary~\ref{cor:2}, we may then conclude that
	\[
	\norm{T_J f}_{L^p(\mi;\ell^\infty(J))}\meg 16\, D_{R}^{5}p'^{1/p}(2D_{R}^{13}+12 C_1 D_{R}^{11} )\norm{f}_{L^p(\mi)}
	\]
	for every $J$  and for every $f\in L^\infty_{c}(\mi)$.  
	
	In order to simplify the notation, we write $p^{(k')}\coloneqq (p_1,\dots, p_{k'})$ and $\nu^{(k')}\coloneqq (\nu_1,\dots, \nu_{k'})$ for every $k'\in \Set{0,\dots, k}$.  
	Now, assume by induction that, for some $k'\in \Set{0,\dots, k-1}$ (with some abuse of notation),
	\[
	\begin{split}
	 &\norm{ [T_J f(\,\cdot\,)(y)](x)}_{L^p_x(\mi;L_y^{p^{(k')}}(\nu^{(k')};\ell^\infty(J)))}\\
     &\quad\meg 16^{k'+1} (2+(k'+1)12 C_1) D_{R}^{18+16k'}  p_1'^{1/p_1} \phi(p_{k'},p)\prod_{j=1}^{k'-1} \phi(p_j,p_{j+1}) \norm{f}_{L^p(\mi;L^{p^{(k')}} (\nu^{(k')}))}
	\end{split}
	\]	
	for every $f\in   L^{p}\big(\mi;L^{p^{(k')}} \big(\nu^{(k')}\big)\big)$ and for every $J$, and let us prove the analogous assertion for $k'+1$ (notice that the case $k'=0$ has been established above).
	To this aim, we shall apply  Corollary~\ref{cor:2}  to $B_1=L^{p^{(k'+1)}} (\nu^{(k'+1)})$, $B_2= L^{p^{(k'+1)}} (\nu^{(k'+1)};\ell^\infty(J))$, $r=p_{k'+1}$, $\kappa=7/3$, $R=9R'$, and to the operator $T$ defined so that
	\[
	T f(x)= \int_X K_J(x,x')f(x')\,\dd \mi(x'),
	\]
	where $[K_J(x,x')f(x')](y)=K_J(x,x') [f(x')(y)]$ for almost every $y$. Note that $K_J(x,y)-K_J(x',y')$ has the same norm as an element of $\ell^\infty(J)$ and as its canonical extension to an  element of $\Lc(B_1;B_2)$ (provided that all $\nu_j$, $j=1,\dots,k'+1$, be non-trivial; otherwise, there is nothing to prove). 
	Observe that, since the $L^{p_{k'+1}}(\mi )$ norm and the $L^{p_{k'+1}}(\nu_{k'+1})$ norm commute, (ii)$_{9R'}$ and (ii)$'_{9R'}$ hold with
    \begin{equation*}
        A_{9R'}=16^{k'+1} (2+(k'+1)12 C_1) D_{R}^{18+16k'} p_1'^{1/p_1}  \prod_{j=1}^{k'} \phi(p_j,p_{j+1}).
    \end{equation*}
	Then, by means of  Corollary~\ref{cor:2}) we see that 
	\[
	\begin{split}
	 &\norm{T_{J} f}_{L^{p}(\mi;L^{p^{(k'+1)}}(\nu^{(k'+1)};\ell^\infty(J)))}\\
     &\quad\meg 16^{k'+2} (2+(k'+2)12 C_1) D_{R}^{18+16(k'+1)}\times\\
     &\qquad\times p_1'^{1/p_1} \phi(p_{k'+1},p)\prod_{j=1}^{k'} \phi(p_j,p_{j+1})   \norm{f}_{L^{p }(\mi;L^{p^{(k'+1)}}(\nu^{(k'+1)}))}
	\end{split}
	\]
	for every $p\in (1,\infty)$ and for every  $f\in   L^{\infty }( \mi;L^{p^{(k'+1)}}(\nu^{(k'+1)}))$.
	
	Now, observe that, for every $j\in \N$, for every $\mi$-measurable function $f\colon X\to \C$ and for every $x \in X$ and $r\in (2^{-j-2}  R', 2^{-j-1}  R']$,
	\[
	\dashint_{B(x,r)} \abs{f}\,\dd \mi\meg D_{  R'} \dashint_{B(x,2^{-j-1}  R')} \abs{f}\,\dd \mi
	\]
	since $\mi(B(x,r))\Meg D_{  R'}^{-1}\mi(B(x,2 r))\Meg D_{  R'}^{-1}   \mi(B(x,2^{-j-1} R'))$.
	Consequently, by the monotone convergence theorem,
	\[
	(\widetilde \cM_{  R'/2} f)(x)\meg D_{  R'}\sup_{j\in\N} \dashint_{B(x,2^{-j-1}R')} \abs{f}\,\dd \mi  \meg   \frac{D_{R'}^7}{C_1} \sup_{J\in \Pc_0(\N)} \norm{T_J f(x)}_{\ell^\infty(J)}
	\]
	where $\Pc_0(\N)$ denotes the set of finite subsets of $\N$. The assertion with $2R$ replaced by $  R'/2=R/126$  then follows by means of the monotone convergence theorem.

	In order to complete the proof, we need to estimate the operator $\cM'$ defined by
	\[
	(\cM' f)(x)\coloneqq   \frac{1}{V_{ R/126 }(x)}\int_{B(x,2R)} \abs{f}\,\dd \mi.
	\]
	Notice that, by Minkowski's integral inequality, we no longer need to deal with mixed-norm spaces.
	Then, it will suffice to apply  Schur's lemma, since
	\[
	\sup_{x\in X} \int_{B(x,2R)} \frac{1}{V_{ R/126}(x)}\,\dd \mi(y)\meg D_R^{8}.
	\]
	The assertion follows. 
\end{proof}

\end{document}